\documentclass{article}
\usepackage[utf8]{inputenc}
\usepackage[T1]{fontenc}
\usepackage{amsmath,amssymb,amsthm,bbm,float,graphicx,geometry,lmodern,mathtools,parskip,setspace,subcaption}
\usepackage{mathrsfs}
\usepackage{enumerate}
\usepackage{nicefrac}
\usepackage{todonotes}
\usepackage{comment}
\usepackage[colorlinks, linkcolor = blue!80!black, citecolor = blue!80!black, breaklinks, pdfauthor={Lukas Luechtrath}]{hyperref}
\usepackage{orcidlink}
\usepackage{titling}

\usepackage{tikz}
\usetikzlibrary{backgrounds}
\usetikzlibrary{patterns}
\usetikzlibrary{positioning, shapes.geometric}
\usetikzlibrary{arrows.meta, calc, decorations.pathmorphing}
\usetikzlibrary{intersections}

\usepackage{caption}
\usepackage{graphicx}
\graphicspath{{Images/}}
\allowdisplaybreaks

\newcommand{\N}{\mathbb{N}}

\newtheorem{theorem}{Theorem}
\newtheorem{lemma}[theorem]{Lemma}
\newtheorem{corollary}[theorem]{Corollary}
\newtheorem{prop}[theorem]{Proposition}

\theoremstyle{definition}
\newtheorem{definition}[theorem]{Definition}
\newtheorem{remark}{Remark}

\usepackage{nameref}

\makeatletter
\let\orgdescriptionlabel\descriptionlabel
\renewcommand*{\descriptionlabel}[1]{%
  \let\orglabel\label
  \let\label\@gobble
  \phantomsection
  \edef\@currentlabel{#1}%
  \let\label\orglabel
  \orgdescriptionlabel{#1}%
}

\renewcommand{\P}{\mathbb{P}}
\newcommand{\x}{\boldsymbol{x}}
\newcommand{\y}{\boldsymbol{y}}
\newcommand{\1}{\mathbbm{1}}
\newcommand{\G}{\mathscr{G}}

\newcommand{\E}{\mathbb{E}}

\newcommand{\R}{\mathbb{R}}
\renewcommand{\d}{\mathrm{d}}
\newcommand{\z}{\mathbf{z}}
\newcommand{\w}{\boldsymbol{w}}

\newcommand{\Z}{\mathbb{Z}}

\renewcommand{\d}{\operatorname{d}}

\newcommand{\scrV}{\mathscr{V}}

\newcommand{\scrE}{\mathscr{E}}
\newcommand{\calE}{\mathcal{E}}
\renewcommand{\L}{\mathcal{L}}

\newcommand{\scrG}{\mathscr{G}}

\pretitle{\centering\LARGE\scshape}
 \posttitle{\vskip 0.75cm}

 \predate{\vskip 0.5 cm \centering\large}
 \postdate{\par}

\usepackage[style = numeric, sorting=nyt, url = false, abbreviate=false, maxbibnames=9, sortcites=false, doi = true, backend = biber, giveninits = true, isbn = false]{biblatex}
\renewbibmacro{in:}{\ifentrytype{article}{}{\printtext{\bibstring{in}\intitlepunct}}}
\DeclareFieldFormat{journaltitle}{\mkbibemph{#1\isdot}}
\title{All spatial random graphs with weak long-range effects have chemical distance comparable to Euclidean distance}

\thanksmarkseries{arabic}

\author{
Lukas L\"{u}chtrath 
	\orcidlink{0000-0003-4969-806X}
	\thanks{
		Weierstrass Institute for Applied Analysis and Stochastics, Mohrenstr.\ 39, 10117 Berlin, Germany. 
		\\ 
		\phantom{E-m }E-mail: lukas.luechtrath@wias-berlin.de
	}
}

\date{November 16, 2025}

\begin{document}
\maketitle

\begin{spacing}{0.9} 
	\begin{abstract} 
		\noindent This note provides a sufficient condition for linear lower bounds on chemical distances (compared to the Euclidean distance) in general spatial random graphs. The condition is based on the scarceness of long edges in the graph and weak correlations at large distances and is valid for all translation invariant and locally finite graphs that fulfil these conditions. We apply the result to various examples, thereby confirming a conjecture on graph distances in the heavy-tailed Boolean model posed by Hirsch [Braz.\ J.\ Probab.\ Stat.\ (2017)]. The proof is based on a renormalisation scheme introduced by Berger [arXiv: 0409021 (2004)]. 

	\smallskip
		\noindent\footnotesize{{\textbf{AMS-MSC 2020}: Primary: 60K35; Secondary: 90B15, 05C80}

	}
	\end{abstract}
\end{spacing}

\section{Introduction and statement of results}
A fundamental question in percolation theory concerns the scaling relationship between the graph distance, often referred to as the \emph{chemical distance}, and the Euclidean distance between two vertices. Understanding this relationship provides insights into the geometric and structural properties of the infinite cluster in supercritical percolation models. Models for which the chemical distances were studied are for instances classical Bernoulli percolation~\cite{AntalPisztora96,Garet_Marchand_2007_chmical}, its continuum analogue~\cite{YaoChenGuo2011}, random interlacements~\cite{CernyPopov_2012}, and the level sets of the Gaussian free field~\cite{DrewitzEtAl_2014_chemDist}. Although some of the models have long-range interactions, these models have in common that edges are of bounded length, and thus the chemical distance of two distant vertices of the infinite component is typically of the same order as their Euclidean distance. Conversely, introducing edges of unbounded lengths can lead to drastically different scaling behaviours, and chemical distances may depend logarithmically or even iterated logarithmically on the Euclidean distances~\cite{Biskup2004, DeijfenHofstadHooghiemstra2013, GGM22, LakisEtal_2024_chemical}. This property may be seen as a spatial version of the famous (ultra) \emph{small world} property of complex networks~\cite{ChungLu2006}. However, unbounded edge lengths do not always lead to significantly shorter graph distances. Consider for instance long-range percolation~\cite{Schulman1983}, in which each pair of lattice sites is connected by an edge with a probability that decays polynomially with power \(-d\delta\) in the vertices' distance. If \(1<\delta<2\), then the graph distance between two distant vertices \(x,y\) is of order \(\log^\Delta|x-y|\) for \(\Delta=1/\log_2(2/\delta)\)~\cite{Biskup2004}, while for \(\delta>2\) the graph distance is given by a linear function of the Euclidean distance~\cite{berger2004,Baeumler2023_continuity}. The reason for this dramatic change of behaviour is simply that for \(\delta>2\) long edges are too rare to give a significant advantage over the bounded edge lengths model.  

This note builds on the work~\cite{berger2004} and demonstrates that such linear scaling is a general feature of spatial graphs where long edges are sufficiently rare. Using the framework of~\cite{jacobJahLu2024}, which relates long edges to subcritical annulus-crossing probabilities in graphs with weak correlations, we establish linear lower bounds on the chemical distance under broad assumptions. Furthermore, we demonstrate that the probability of shorter chemical distances at large distances decays at the same rate at which long edges are present, which also provides a new quantitative result for long-range percolation. 

\paragraph{Framework.}
We use the framework of~\cite{jacobJahLu2024} and consider general translation-invariant models defined on some appropriate probability space whose probability measure we denote by \(\P\). We aim to study a countably infinite random graph \(\G=(\mathscr{V},\mathscr{E})\) whose vertex set \(\scrV\) is built on the points of a suitable point process on \(\R^d\). We make the following assumptions throughout:
\begin{description}
 	\item[(G1)\label{G:Point_process}] 
 		The locations of the vertices in \(\scrV\) are given by a \emph{simple point process} that is \emph{translation invariant}. Note that this refers to the location of the vertices only and the vertices may carry additional marks or weights. In fact, in most examples, vertex marks are used to model a vertex' attraction or influence in some way. Throughout the manuscript, we denote vertices by \(\x\in\scrV\). For a vertex \(\x\in\scrV\), we denote its location by \(x\in\R^d\). Although \(\scrV\) refers to the whole vertex set and may contain additional markings, we still write \(\x\in\scrV\cap A\) for a vertex \(\x\) with location \(x\in A\subset\R^d\).  
 	\item[(G2)\label{G:Translation}] 
 		We assume that \(\G\) is \emph{translation invariant}. That is, \(\G\) and \(\G+x\) have the same distribution, where \(\G+x\) is the graph constructed on the points of \(\scrV+x\), i.e., the connection mechanism does not change if the location of each vertex is shifted. 	
 	\item[(G3)\label{G:locallyFinite}] 
 		The graph \(\scrG\) is \emph{locally finite}, i.e., all vertices have finite degree, almost surely. 
 \end{description}

\begin{remark} ~\
	\begin{enumerate}[(a)]
		\item Our setup essentially includes two types of underlying vertex locations that can simultaneously be dealt with. First, standard stationary point processes, i.e., processes with intensity measure given by the Lebesgue measure multiplied by an intensity parameter \(\lambda>0\)~\cite{LastPenrose2017}. Secondly, locations that are based on some site-percolation process on the lattice \(\Z^d\). In this case, we define \(\lambda=\P(o\in\scrV)\) where \(o\) denotes the origin. In both cases, \(\lambda\) denotes a classical percolation parameter. As studying the graph distance of far apart vertices is only reasonable if there is an infinite connected component containing both vertices, one may always assume that \(\lambda\) is chosen large enough to guarantee supercriticality. However, we only prove a lower bound on the graph distance so the two vertices being in the same component is not strictly necessary. Therefore, \(\lambda\) plays no particular role in the following and we thus drop it from the notation. 
		\item The two standard examples for the underlying point process that are most frequently used in the literature are the homogeneous Poisson point process and the Bernoulli site-percolated lattice. Other examples are discussed in~\cite{jacobJahLu2024}.
	\end{enumerate}
\end{remark}

\paragraph{Quantifying long-range effects.}
Next, we formulate the properties that are crucial for our main result, which are subject to the long-range effects of the graph. Typically, there are two ways in which these interactions can arise. One is through the presence of long edges in the graph that connect vertices far apart from each other. The other is via correlations of the local configurations of the graph across distant regions. While the amount of long edges is most relevant in order to deduce the typical graph distances' scale, we still require some control over the influence of the latter type on the graph topology. To employ our proof, we require both effects to vanish on a polynomial scale. More precisely, {write \(\Lambda_m(x)=x+[-m/2,m/2)^d\) for the box of side length \(m\), centred in \(x\), and} let \(\Lambda_m(o)\) and \(\Lambda_m(m x)\) be two such boxes, centred in the origin \(o\) and in \(m x\) for some \(|x|>2\) respectively. Note that the two boxes are at distance at least \(m\). Let \(\calE(\Lambda_m(o))\) be a \emph{local event}, meaning that it can be decided by the realisation of the internal vertices and edges of \(\Lambda_m(o)\) alone. { Put differently, \(\calE(\Lambda_m(o))\) is a local event if
	\[
		\E[\1_{\calE(\Lambda_m(o))}\mid \scrG\cap \Lambda_m(o)] \in\{0,1\}.
	\]
	A standard example for a local event is the existence of a connected component of a certain size within the box \(\Lambda_m\).} We say that \(\scrG\) is \emph{polynomially mixing} with mixing exponent \(\xi<0\) if, {for every local event \(\calE(\cdot)\)}, there exists a suitable constant \(C_\text{mix}>0\) such that for sufficiently large \(m\), 
    \[ \tag{\({{\mathscr{P}}\!\!{\mathscr{M}}}\)} \label{G:Pmix}
    	{\sup_{|x|>2}}\operatorname{Cov}\big(\1_{\calE(\Lambda_m(o))},\1_{\calE(\Lambda_m(mx))}\big)\leq C_\text{mix} \,  m^{\xi}.
    \]
    If \(\scrG\) is polynomially mixing with exponent \(\xi\), we also say that \(\scrG\) has Property~\ref{G:Pmix}\(^\xi\). {Note that the constant \(C_\text{mix}\) may apriori depend on the event but, importantly, it does not depend on \(m\). In practice, we apply this property only to the `bad box’ event used in the renormalisation scheme of Section~\ref{sec:proofMain}, so we may take a single constant throughout.}
    
    Let us further quantify the occurrence of long edges. To this end, define, for \(n,m\in\N\), the event 
 \[
    \L(m,n):= \big\{\exists \x\sim \y: x,y\in\Lambda_m(o) \text{ and }|x-y|>n\big\},
 \]
 where \(\x\sim \y\) denotes the existence of an edge connecting the vertices \(\x\) and \(\y\). Therefore, \(\L(m,n)\) describes the occurrence of long edges in a box. Henceforth, we say that \(\scrG\) has the `no long edges' Property~\ref{G:NoLongE}\(^\mu\), if there exists an exponent \(\mu<-d\) and a constant \(C_\L>0\) such that for all \(m\geq 1\) and sufficiently large \(n\), we have 
 \begin{equation*}\tag{\(\mathscr{P}\!\!\mathscr{L}\)}\label{G:NoLongE}
 	\P(\L(m,n))\leq C_\L m^d n^{\mu}.
 \end{equation*}
 
\paragraph{Main result.} 
We denote the graph distance in the graph \(\G\) by \(\operatorname{d}:=\operatorname{d}_\G\), i.e., 
\[
	\operatorname{d}(\x,\y):=\operatorname{d}_\G(\x,\y)=\min\big\{m: \exists \text{ path }\x=\x_0\sim \x_1\sim\dots\sim \x_m=\y \text{ in } \G\big\} \quad \text{ for } \x,\y\in\G,
\]
with the usual convention \(\min\emptyset=\infty\). To formulate our main theorem, let us define the event
\[
 \mathcal{D}_{L}^\eta(m) = \big\{\operatorname{d}(\x,\y)\geq \eta|x-y|, \text{ for all } {\x\in\Lambda_L (o)\text{ and }\y\in\Lambda_{m}^\mathsf{c}(o)}\big\},
\]
stating that all paths connecting a vertex located relatively close to the origin to a faraway vertex have length lower bounded by a constant multiple of the Euclidean distance of their end-vertices.

\medskip

\begin{theorem}[Linear graph distances, lower bound]\label{thm:main}
    If \(\G=(\scrV,\scrE)\) has the Properties~\ref{G:Pmix}\(^\xi\) and~\ref{G:NoLongE}\(^\mu\), { for some \(\xi<0\) and \(\mu<-d\),} then there exists a constant \(\eta>0\), depending only on model parameters, such that for all \(L\in\N\), 
    \begin{equation}\label{eq:mainTail}
    	\limsup_{m\to\infty}\frac{\log \P\big(\neg\mathcal{D}^\eta_{L}(m)\big)}{\log m} \leq \xi \vee (d+\mu).
    \end{equation}
\end{theorem}

\medskip

	Let us, for the moment, assume that \(d+\mu>\xi\). Then, Theorem~\ref{thm:main} essentially states that the probability of finding a path that leaves a box of side length \(m\) with `only a few steps' decays at least at the same rate as \(\P(\L(m,m))\), which is the probability of having at least one edge of length \(m\) in said box. Put differently, either all paths contain a number of edges proportional to the box' side length, or there is one single edge spanning the whole Euclidean distance alone. A similar behaviour was observed in~\cite{jacobJahLu2024}, where the existence of a subcritical annulus-crossing phase, i.e.\	
	\[
		\widehat{\lambda}_c:=\inf\big\{\lambda>0: \lim_{m\to\infty}\P(\text{annulus with radii }m \text{ and }2m \text{ is crossed by a path})>0\big\}>0,
	\]
	was investigated. In that article, the authors consider the `long edges' event \(\L(m):=\L(m,m)\) and show that, if \(\P(\L(m))\to 0\) as \(m\to\infty\) uniformly in small intensities \(\lambda\), then \(\widehat{\lambda}_c>0\), provided the model mixes in the sense that the defining covariance in~\eqref{G:Pmix} tends uniformly to zero. Conversely, if the convergence of \(\P(\L(m))\) does not hold for any intensity \(\lambda\), then there is no such phase, and annuli are typically crossed by a single long edge. The reason we require a slightly finer version of the `long edges' event and the polynomial decay lies in the renormalisation scheme of Berger~\cite{berger2004}, our proof is based on. In order to employ it, we must exclude the existence of long edges that are on a smaller scale than the boxes in which they are contained. However, both versions of the event are closely related. Clearly, Property~\ref{G:NoLongE}\(^\mu\) implies polynomial decay of \(\P(\L(m))\), and furthermore, polynomial decay of \(\P(\L(m))\) also implies Property~\ref{G:NoLongE}\(^\mu\) in all our examples below. Let us additionally note that, except for boundary cases, typical models of our interest either satisfy~\eqref{G:NoLongE} with a polynomial decay or contain many long-edges. 	
	
	Although the decay of the mixing term of Property~\ref{G:Pmix}\(^\xi\) could also be the dominant term in~\eqref{eq:mainTail}, we believe this to be a technical result and only the presence of long edges is truly significant. It is demonstrated in~\cite{DrewitzEtAl_2014_chemDist} that models with polynomial correlations, which are monotone in the intensity, can still exhibit \emph{weak correlations for monotone events} in the sense that monotone events in disjoint regions are independent after a sprinkling, depending on the Euclidean distance of the regions, with a stretched exponential error term. If the required sprinkling is summable over the scales, one should be able to adapt our proof in a way that the long edges always yield the dominant contribution. Note, however, that this property must be proved individually for each model, which is why we work with the weaker assumption of polynomially bounded correlations. 
	
	Finally, if one sets \(a=2/|\xi\vee (d+\mu)|\) and \(m_k=\lceil k^a\rceil\), then \(\P(\neg\mathcal{D}_{L}(m_k))\) is summable by Theorem~\ref{thm:main}. Thus, by the Borel-Cantelli lemma, almost surely, all paths connecting the vertices in \(\Lambda_L\) to a large distance become bounded from below by a constant multiple of their end-vertices' Euclidean distance eventually. This recovers the original result of Berger for long-range percolation~\cite{berger2004} and extends it to the whole class of models that satisfy our assumptions.   
	 	 
\begin{remark}[On the upper bound] \label{rem:upper} 	
	Theorem~\ref{thm:main} only establishes a linear lower bound on graph distances in terms of the Euclidean distance. The natural question is whether there exists a corresponding upper bound. {We believe this to be true for percolation models that are sufficiently decorrelated at large distances; in particular those discussed in Section~\ref{sec:Examples}. } However, deriving this is generally challenging. Let us only consider models constructed on either a Poisson process or a site-percolated lattice, in which edges do only depend on their end vertices to ensure a sufficient amount of independence. For long-range percolation on the lattice, the upper bound and the corresponding shape-theorem were only recently proven in~\cite{Baeumler2023_continuity}. The proof relies on the continuity of the critical value, a property that is generally hard to prove. This continuity allows the removal of edges longer than an appropriate threshold, hence approximating the graph by one with bounded edge lengths. For this approximating graph, one can adapt the established results of~\cite{AntalPisztora96,Cerf_Theret_2016_fpp} to conclude the proof. Similarly, in~\cite{DrewitzEtAl_2014_chemDist} linear upper bounds for correlated percolation models on the lattice are proven in general, provided the model exhibits the aforementioned weak correlations for monotone events and local uniqueness of the infinite component. The latter property is again closely related to the continuity of the critical value. However, this continuity has only been established in regimes with many long edges~\cite{Moench2024}, and not in our regime, where long edges are rare. In a nutshell, long edges are too rare to use them constructively as in~\cite{Moench2024} but they still prevent us from applying localisation arguments from finite-range models.
	  
	It is worth noting that many examples often dominate either a long-range percolation model, a Gilbert graph or a model in the framework of~\cite{DrewitzEtAl_2014_chemDist}. If these underlying models are supercritical, one immediately observes linear upper bounds. Hence, in these examples chemical distances are linear in the Euclidean distance for sufficiently large intensities \(\lambda\gg\lambda_c\), and we believe this to be true for all \(\lambda>\lambda_c\). 
	\end{remark}      
    
\section{Examples}\label{sec:Examples}
   Let us apply our result to a couple of examples. We shall only briefly sketch these examples and refer the reader to \cite{jacobJahLu2024} for a more detailed discussion on the models.
   
  \paragraph{The weight-dependent random connection model.}   
  	This model was first introduced in~\cite{GHMM2022,GLM2021} as a general framework for inhomogeneous random connection models. It contains many models from the literature as special instances, cf.~\cite[Tab.~1]{GHMM2022}. In this model, \(\scrV\) is either a standard Poisson point process or a (Bernoulli site-percolated) lattice where each vertex carries additionally an independent uniformly distributed vertex mark. We denote a vertex by \(\x=(x,u_x)\in\R^d\times (0,1)\). Here, the mark \(u_x\) models the \emph{inverse weight} of a vertex, thus the smaller \(u_x\) the more attractive for connections the vertex \(\x\) is. Currently, the model is most frequently studied using the following connection rule~\cite{GraLuMo2022}: Given \(\scrV\), each pair of vertices \(\x\) and \(\y\) is connected by an edge independently with probability
  \begin{equation}\label{eq:WDRCM}
  	  \rho \big((u_x\wedge u_y)^\gamma(u_x\vee u_y)^{\gamma'} |x-y|^d\big), \qquad \text{ where }\gamma\in[0,1) \text{ and }\gamma'\in[0,2-\gamma).
  \end{equation}
   Here, \(\rho\) is a decreasing and integrable \emph{profile function}, typically chosen to be either \(\rho(x)\asymp 1\wedge |x|^{-\delta}\) for some \(\delta>1\) or \(\rho=\1_{[0,1]}\) {where the notation \(f\asymp g\) indicates that \(f/g\) is bounded from zero and infinity.} If one identifies the indicator function case with \(\delta=\infty\), one can identify and compare many models from the literature by only three real parameters~\cite[Tab.~1]{GraLuMo2022}.
      
   It is clear that local events in disjoint regions are independent in this setting, and thus \ref{G:Pmix}\(^{-\infty}\) is satisfied. It therefore remains to check for which choices of parameters Property~\ref{G:NoLongE}\(^\mu\) holds in order to apply Theorem~\ref{thm:main}. To this end, the \emph{downwards vertex boundary} exponent \(\zeta\), introduced in~\cite{JorritsmaMitscheKomjathy2023}, can be used. Loosely speaking, \(\zeta\) quantifies the number of vertices within a box \(\Lambda_m\) that are connected to a vertex at distance \(m\), which is weaker (i.e.\ has larger mark) than themselves (thus the term `downward'). If \(\zeta<0\), then there are only a few such vertices, and long edges are rare, {while for \(\zeta>0\) long edges are typical. More precisely, it is shown in~\cite{JorritsmaMitscheKomjathy2023,jacobJahLu2024} that \(\P(\L(m,m))\asymp m^{d\zeta}\) if \(\zeta<0\), and \(\P(\L(m,m))\geq 1-e^{-m^{d\zeta}}\) if \(\zeta>0\). We next show that \ref{G:NoLongE}\(^{d(\zeta-1)}\) is satisfied in the regime where \(\zeta<0\).}
   {
   \begin{lemma}\label{lem:noLongE}
   	Consider the weight-dependent random connection model~\eqref{eq:WDRCM}. If \(\delta>2\), \(\gamma<1-1/\delta\), and \(\gamma'<1-\gamma\), then \ref{G:NoLongE}\(^{d(\zeta-1)}\) is satisfied for
   	\begin{equation}\label{eq:zeta}
   		\zeta = \max\big\{2-\delta, 1-\tfrac{\delta-1}{\gamma\delta}, \tfrac{\gamma'+\gamma-1}{\gamma},\tfrac{2(\gamma'+\gamma-1)}{\gamma'+\gamma}\big\} <0
   	\end{equation}	
   \end{lemma}
   \begin{proof}
   		Let \(m,n>1\) and consider
   		\[
   			\P(\L(m,n)) \leq \sum_{k=1}^\infty \P\big(\exists \x\sim\y\colon \x\in\Lambda_m(o), kn\leq |x-y|<(k+1)n\big)=:\sum_{k=1}^\infty \P(\L_k(m,n)).
   		\] 
   		Consider some fixed \(k\in\N\) and let \(u_{kn}\in(0,1)\) be the largest mark for which
   		\[
   			\E_{(o,u_{kn})}\big[\sharp\{\y\colon |y|\geq kn, u_y>u_{kn}, (o,u_{kn})\sim \y\}\big]\geq 1,
   		\] 
   		where \(\E_{(o,u_{kn})}\) denotes the expectation, given the existence of the vertex \((o,u_{kn})\). Put differently, \(u_{kn}\) is the `weakest' required mark a vertex must have in order to connect, in expectation, to at least one larger-mark vertex at distance \(kn\). By \cite[Sec.~2.1 and Prop.~2.3]{jacobJahLu2024} or~\cite{JorritsmaMitscheKomjathy2023}, respectively, we have, for the considered parameter regime, \(u_{kn}=(kn)^{d(\zeta-1)}\), for \(\zeta\) given in~\eqref{eq:zeta}. Hence,
   		\[
   			\begin{aligned}
   				\P(\L_k(m,n))\leq \P(\exists
   				&
   					\x\in\Lambda_m(o)\colon u_x\leq (kn)^{d(\zeta-1)}) 
   				\\&
   					+ \P\big(\exists \x\sim\y\colon \x\in\Lambda_m(o), kn\leq|x-y|<(k+1)n, u_x,u_y>(kn)^{d(\zeta-1)}\big).
   			\end{aligned}
   		\]
   		In words, for \(\L_k(m,n)\) to occur, there must either exist a vertex with a sufficiently strong mark or there exists an edge between two relatively weak vertices. By a standard expectation bound for the Poisson point process, the first event has probability bounded by \(m^d (kn)^{d(\zeta-1)}\). Applying Markov's inequality, Mecke's equation~\cite{LastPenrose2017}, and the fact that \(\rho\) is a decreasing function, the second probability on the right-hand side is bounded by
   		\[
   			\begin{aligned}
   				\int\limits_{\Lambda_m(o)} 
   				& 
   					\d x \int\limits_{(kn)^{d(\zeta-1)}}^1 \d u_x \int\limits_{\substack{(k+1)n>|x-y|\geq kn}} \d y \int\limits_{(kn)^{d(\zeta-1)}}^1 \d u_y \, \rho \big((u_x\wedge u_y)^\gamma(u_x\vee u_y)^{\gamma'} |x-y|^d\big)
   				\\ &
   					\leq C m^d (kn)^d \int\limits_{(kn)^{d(\zeta-1)}}^1 \d u_x \int\limits_{(kn)^{d(\zeta-1)}}^1 \d u_y  \, \rho \big((u_x\wedge u_y)^\gamma(u_x\vee u_y)^{\gamma'} (kn)^d\big)
   				\\ &
   					= C m^d (kn)^{-d}\Big( (kn)^{2d}\int\limits_{(kn)^{d(\zeta-1)}}^1 \d u_x \int\limits_{(kn)^{d(\zeta-1)}}^1 \d u_y  \, \rho \big((u_x\wedge u_y)^\gamma(u_x\vee u_y)^{\gamma'} (kn)^d\big)\Big),
   			\end{aligned}
   		\]
  		for some constant \(C>1\). However, by~\cite[Sec.~2.1 and Eq.~(5)]{jacobJahLu2024}, the term inside the brackets is of the order \((kn)^{d\zeta}\), and therefore   
  		\[
  			\sum_{k\geq 1}\P(\L_k(m,n))\leq Cm^d n^{d(\zeta-1)}\sum_{k\geq 1} k^{d(\zeta-1)}\leq Cm^d n^{d(\zeta-1)},
  		\]
  		since \(k^{d(\zeta-1)}\) is summable. This concludes the proof.
  	\end{proof}
   	
   	Let us remark that we additionally infer from~\cite{JorritsmaMitscheKomjathy2023,jacobJahLu2024} that \(\zeta\geq 0\) for all parameter choices other than the ones of Lemma~\ref{lem:noLongE} and that~\eqref{G:NoLongE} cannot be satisfied in such a case. Furthermore, Theorem~\ref{thm:main} and Lemma~\ref{lem:noLongE} immediately imply the following result. 
    }
   
   \begin{corollary}\label{cor:WDRCM}
   	Consider the weight-dependent random connection model~\eqref{eq:WDRCM}. If \(\delta>2\), \(\gamma<1-1/\delta\), and \(\gamma'<1-\gamma\), then   
   	\[
   		\limsup_{m\to\infty}\log\big(\P\big(\neg\mathcal{D}^\eta_{L}(m)\big)\big)/\log(m) \leq d\zeta,
   	\]
   	with \(\zeta\) as in~\eqref{eq:zeta}.
   	\end{corollary}
   
   Note that the model~\eqref{eq:WDRCM} always dominates a Gilbert graph and a linear upper bound on the chemical distances follows by domination for all large enough \(\lambda\), cf.\ Remark~\ref{rem:upper} and~\cite{YaoChenGuo2011}. Let us consider two models from the class of weight-dependent random connection models more closely.
   
 \emph{The Boolean model} corresponds to the choice of \(\gamma>0\), \(\gamma'=0\), and \(\delta=\infty\)~\cite{Hall85}. That is, each vertex is assigned an independent Pareto distributed radius with tail exponent \(d/\gamma\) and each vertex is the centre of a ball of its associated radius. Any two vertices are connected by an edge if the centre of the smaller ball is contained in the larger ball. It was shown in~\cite{Hirsch2017} that the chemical distances are lower bounded by \(|x|/\log^p|x|\) for any power \(p>0\). However, it was conjectured that the logarithmic correction was a result of the method and that graph distances are indeed linearly lower bounded, cf.~\cite[Conj.~5]{Hirsch2017}. Corollary~\ref{cor:WDRCM} proves the conjecture. 
 
 \emph{The soft Boolean model} corresponds to the choice of \(\gamma>0\), \(\gamma'=0\), and \(\delta\in(1,\infty)\)~\cite{GGM22}, which is a model that interpolates between the Boolean model and long-range percolation~\cite{jahnel_Lu_Ort_2024_cluster}. It was shown in~\cite{GGM22} that chemical distances are doubly logarithmic in the Euclidean distance for all \(\delta>1\) and \(\gamma>1-1/(\delta+1)\). Further,~\cite{GGM22} elaborates that distances are at least logarithmic in the Euclidean distance if \(\gamma<1-1/(\delta+1)\). We find with Corollary~\ref{cor:WDRCM} that graph distances are at least linear if \(\delta>2\) and \(\gamma<1-1/\delta\), specifying the result of~\cite{GGM22}. Recent results for scale-free percolation (\(\gamma=\gamma'>0\)) suggest that for \(\delta<2\) and \(\gamma<1/2\), graph distances are a power of the logarithm of the Euclidean distance, where the exponent is the one of long-range percolation~\cite{LakisEtal_2024_chemical}. It remains an interesting open problem to identify the power of the logarithm for the regime \(\delta>2\) and \(1/2<\gamma<1/(\delta+1)\). We leave this for future work. 
 
 \paragraph{Soft Boolean model with local interference.}
 This model is an example for a \emph{generalised weight-dependent random connection model}, in which edges not only depend on their end-vertices but also on surrounding vertex clouds~\cite{jacobJahLu2024}. The vertex set is the same as for the soft Boolean model and a vertex \(\x=(x,u_x)\) still has assigned a sphere of influence of radius \(u_x^{-d/\gamma}\) that is used together with additional long-range effects to form connections. Additionally, the vertex has assigned a sphere of interference of radius \(u_x^{-d/\beta}\) for some \(\beta<1\). The vertices located in the latter sphere interfere and make it harder for the vertex to form connections. More precisely, two vertices \(\x=(x,u_x)\) and \(\y=(y,u_y)\) with \(u_x<u_y\) are connected by an edge with probability
 \[
 	\mathbf{p}(\x,\y,\scrV) = \frac{1\wedge (u_x^{\gamma}|x-y|^d)^{-\delta}}{\sharp\{\z\in\scrV: |x-z|^d<u_x^{-\beta}\}}.
 \]
 It is elaborated in~\cite{jacobJahLu2024} that the exponent \(\zeta\) can straight-forwardly be generalised to such a setting. It is further shown that the model mixes polynomially with exponent \(\xi=1-1/\beta\) and that it has \(\zeta<0\) if \(\gamma<(\delta+\beta-1)/\delta\). Hence, Theorem~\ref{thm:main} applies in this case and proves linear lower bounds on the chemical distance, with probabilistic rate no slower than polynomially with exponent \(\xi\vee \zeta\). 
 
 \paragraph{Ellipses percolation.}  
 Introduced in~\cite{Teixeira_Ungaretti_2017}, this model can be seen as a generalisation of the planar Boolean model. Instead of a ball, each vertex is the centre of an ellipsis with \(\operatorname{Pareto}(2/\gamma)\) distributed major axis, minor axis equal to one, and uniformly distributed orientation. Intriguingly, replacing balls by ellipses introduces additional correlations with interesting effects. The paper~\cite{Hilario-Ungaretti-2021} elaborates graph distances in the regime \(\gamma\in(1,2)\) and shows that these scale doubly logarithmically in the Euclidean distance. In the \(\gamma<1\) regime however Property~\ref{G:NoLongE}\(^{-2/\gamma}\) is satisfied, cf.~\cite{jacobJahLu2024}, and Theorem~\ref{thm:main} shows linear lower bounds on the graph distances. For a more detailed discussion on the correlations involved, their interesting effects, and generalisations of the model, we refer to~\cite{GKM2024_robustness, gracar2025_chemical}. {Specifically,~\cite{gracar2025_chemical} generalises the results of~\cite{Hilario-Ungaretti-2021} on the doubly-logarithmic regime to general convex grains and higher dimensions.}

 \paragraph{More general underlying vertex locations.}
 Theorem~\ref{thm:main} also applies to models that are constructed on more correlated point clouds than a Poisson point process or a Bernoulli site-percolated lattice. For instance one could consider models constructed on a Cox process, a Gibbs process, or a lattice based on percolation of worms~\cite{rath_rokob_2022_worm}, which is done in greater detail in~\cite{jacobJahLu2024}.

\section{Proof of Theorem~\ref{thm:main}} \label{sec:proofMain}
In this section, we employ the renormalisation scheme that goes back to~\cite{berger2004} that we use to proof Theorem~\ref{thm:main}. {All auxiliary results stated in the course of this proof are proven separately in Section~\ref{sec:proofAux}.} We start by defining the scales. Fix a large even \(K\in 2\N\) to be specified later and set \(K_0:=K\) as well as \(K_n:=K(n!)^2\). A stage-\(n\) box is a box \(B_n(x):= \Lambda_{K_n}(x)=x+[-K_n/2,K_n/2)^d\) of side length \(K_n\), centred at \(x\in\R^d\). We abbreviate \(B_n=B_n(o)\). Let us define the notion of a \emph{good} box; we refer to a box that is not good as being \emph{bad}, see Figure~\ref{fig:badBoxes}.
  	
  	\begin{definition}[Good boxes]\label{def:GoodBox}
        Let \(x\in\R^d\) and consider the boxes \(B_n(x)\). We say that
        \begin{enumerate}[(i)]
            \item the stage-\(0\) box \(B_0(x)\) is \emph{good}, if it contains no internal edge longer than \(K_0/100\).
            \item We say that the stage-\(n\) box \(B_n(x)\) is \emph{good} if all \(3^d\) boxes \(B_n^\mathbf{j}(x):=B_n(x+ \tfrac{K_{n-1}}{2}\mathbf{j})\), for \(\mathbf{j}\in\{-1,0,1\}^d\), satisfy the following two conditions:
            \begin{enumerate}[(a)]
                \item No edge internal to \(B_n^\mathbf{j}(x)\) is longer than \(K_{n-1}/100\).
                \item Of all the \(n^{2d}\) stage-\((n-1)\) boxes contained in \(B_n^\mathbf{j}(x)\), no more than \(3^d\) are bad themselves.
            \end{enumerate}
        \end{enumerate}
    \end{definition}
    
    \begin{figure}[t!]
	\begin{center}
	\begin{subfigure}[t]{0.45\textwidth}
			\resizebox{\textwidth}{!}{
				\begin{tikzpicture}[scale=1, every node/.style={scale=0.6}]
					\def\N{12}          
					\def\a{0.4}        
					\def\shift{0.5*\a} 

					\coordinate (A) at (0,0);             
					\coordinate (B) at (\shift,0);        

					\newcommand{\FillCells}[1]{%
  						\foreach \p in {#1} {%
    						\pgfmathparse{\p}%
  						}%
  						\foreach \i/\j in {#1} {%
    						\path[fill=gray!50] ($(B)+(\i*\a,\j*\a)$) rectangle ++(\a,\a);
  						}%
					}
 					\FillCells{7/11}
					\draw[line width=0.8pt] (A) rectangle ++(\N*\a,\N*\a);

					\foreach \k in {1,...,\numexpr\N-1\relax} {
 				 		\draw ($(A)+(\k*\a,0)$) -- ++(0,\N*\a);
  						\draw ($(A)+(0,\k*\a)$) -- ++(\N*\a,0);
					}

					\draw[line width=0.8pt, dotted] (B) rectangle ++(\N*\a,\N*\a);

					\foreach \k in {1,...,\numexpr\N-1\relax} {
  						\draw[dotted] ($(B)+(\k*\a,0)$) -- ++(0,\N*\a);
  						\draw[dotted] ($(B)+(0,\k*\a)$) -- ++(\N*\a,0);
					}
					\draw (0.9 , 2.13) node[circle, fill = black, scale=0.4] {};
					\draw (1.74 , 1.91) node[circle, fill = black, scale=0.4] {};
					\draw[bend angle = 45, bend left, thick] (0.9 , 2.13) to (1.74 , 1.91);
					
					\draw (3.1 , 4.52) node[circle, fill = black, scale=0.2] {};
					\draw (3.27 , 4.6) node[circle, fill = black, scale=0.2] {};
					\draw[bend angle = 45, bend right, thick] (3.1 , 4.52) to (3.27 , 4.6);
				\end{tikzpicture}
			}
			\caption{A box that fails Property~(ii)~(a).}
			\label{subfig:badEdge}
		\end{subfigure}
		\hfill 
		\begin{subfigure}[t]{0.45\textwidth}
			\resizebox{\textwidth}{!}{
				\begin{tikzpicture}[scale=1]
					\def\N{12}          
					\def\a{0.4}        
					\def\shift{0.5*\a} 

					\coordinate (A) at (0,0);             
					\coordinate (B) at (\shift,0);        

					\newcommand{\FillCells}[1]{%
  						\foreach \p in {#1} {%
    						\pgfmathparse{\p}%
  						}%
  						\foreach \i/\j in {#1} {%
    						\path[fill=gray!50] ($(B)+(\i*\a,\j*\a)$) rectangle ++(\a,\a);
  						}%
					}
 					\FillCells{3/4,3/5,3/6,3/7,4/4,4/5,4/6,5/4,5/5,5/6}

					\draw[line width=0.8pt] (A) rectangle ++(\N*\a,\N*\a);

					\foreach \k in {1,...,\numexpr\N-1\relax} {
 				 		\draw ($(A)+(\k*\a,0)$) -- ++(0,\N*\a);
  						\draw ($(A)+(0,\k*\a)$) -- ++(\N*\a,0);
					}

					\draw[line width=0.8pt, dotted] (B) rectangle ++(\N*\a,\N*\a);

					\foreach \k in {1,...,\numexpr\N-1\relax} {
  						\draw[dotted] ($(B)+(\k*\a,0)$) -- ++(0,\N*\a);
  						\draw[dotted] ($(B)+(0,\k*\a)$) -- ++(\N*\a,0);
					}
				\end{tikzpicture}
			}
			\caption{A box that fails Property~(ii)~(b).}
			\label{subfig:badSubbox}
		\end{subfigure}
	\end{center}
	\caption{The bad box \(B_n\), its sub boxes, and, dotted, the right-shifted box \(B_n^{(1, 0)}\) with its sub boxes. In grey, the bad sub boxes of \(B_n^{(1, 0)}\). On the left, two long edges are shown. The longer edge is too long so that \(B_n\) fails Property~(ii)~(a). The second longest edge is short enough on scale-\(n\) but, on scale-\((n-1)\), it creates a bad sub box of the right-shifted box. On the right, \(B_n\) fails Property~(ii)~(b) as the right-shifted box contains \(3^d+1\) bad sub boxes.}
	\label{fig:badBoxes}
	\end{figure}
    
    Our first result is to determine the probability that a box centred at the origin is bad. Define
    \[
    	\psi_K(n):=\P(B_n \text{ is bad}).
    \]  

    \begin{lemma}\label{lem:badSummable} 
    	If \(\G\) has the Properties~\ref{G:Pmix}\(^\xi\) and~\ref{G:NoLongE}\(^\mu\), for some \(\xi<0\) and \(\mu<-d\), then, {for all \(c>4(d+|\xi\wedge(d+\mu)|)\)}, there exists \(\overline{K}\in\N\) such that, {for all \(K \geq \overline{K}\)},
    	\[
    		\psi_K(n)\leq \big((n+1)!\big)^{-2|\xi\vee(d+\mu)|+c/n}.
    	\]
    \end{lemma}    
    
   
    \begin{corollary}\label{corol:badBoxes} 
    	Under the assumptions of Lemma~\ref{lem:badSummable}, there exist constants \(C,\overline{N}\in\N\), such that, for all \(n\geq \overline{N}\) and \(K>\overline{K}\), we have
    	\[
    		\sum_{k\geq n-1}\psi_K(k)\leq C (n!)^{-2|\xi\vee(d+\mu)|+c/n}
    	\] 
    \end{corollary}
    To make use of good boxes, we formulate the following proposition, which is a stronger version of~\cite[Lemma 2]{berger2004} in the spirit of~\cite[Proposition 2.6]{komjathy2023FPP} for cost distances in one-dependent first-passage percolation, and allows us to circumvent the use of the subadditive ergodic theorem. Note that the following result is a deterministic statement on a fixed configuration fulfilling certain assumptions and does not involve any probabilities. 
    
    \begin{prop}\label{prop:Renorm}
    	Fix \(N\geq (2d+1)9^d\). Consider \(\x,\y\in\scrV\) such that \(y\in B_n(x)\) and \(|x-y|>K_{n-1}/8\) for some \(n\geq N\). {On the event that both boxes \(B_n(x)\) and \(B_{n-1}(x)\) are good, there exists a universal constant \(C_1=C_1(d,N)\),} such that each path connecting \(\x\) and \(\y\) that is completely contained in \(B_n(x)\) consists of at least \(C_1|x-y|\) many edges. 
    \end{prop}
    
    {We have now set the stage for proving Theorem~\ref{thm:main}.}
    
    \begin{proof}[Proof of Theorem~\ref{thm:main}] {In the following, we abbreviate \(\Lambda_r:=\Lambda_r(o)\) for any box centred at the origin.} Fix \(L\in\N\) {and some even \(K>\overline{K}\) so that Lemma~\ref{lem:badSummable} and Corollary~\ref{corol:badBoxes} apply}, and choose some \(M\geq 2K\). For \(m\in\N\), let \(n=n_K(m)\) be the largest integer satisfying {\(K_n=K(n!)^2\leq m\)}. As \(m\) (and therefore \(n\)) can be chosen arbitrarily large, we assume \(\Lambda_L\subset B_{n-3}\). Consider now a realisation \(\omega\), in which all the boxes \(B_{n-1}, B_n,\dots\) are good. Let \(\x\in\Lambda_L\) and \(\y\in B_{n}^\mathsf{c}\) be two vertices in this realisation. Let \(n_y\) be the first index such that \(\y\in B_{n_y}\setminus B_{n_y-1}\). Particularly, this implies \(n_y\geq n\) as well as \(\sqrt{d}(K_{n-1}-L)/2<|x-y|<\sqrt{d}(K_{n_y}+L)/2\). Let further \(\pi\) be a path from \(\x\) to \(\y\) and let \(n_\pi\geq n_y\) be the first index such that \(\pi\subset B_{n_\pi}\). By definition of \(n_\pi\), there exists a vertex \(\boldsymbol{v}\in\pi\cap B_{n_{\pi}}\setminus B_{n_\pi -1}\). If \(n_\pi=n_y\), we pick \(\boldsymbol{v}=\y\). Otherwise, \(n_\pi>n_y\) and we pick the first \(\boldsymbol{v}\) in \(\pi\) with this property, and consequently,
	\[
		|v-x|\geq (K_{n_\pi-1}-L)/2 \geq (K_{n_y}-L)/2 \geq |y-x|/\sqrt{d}.
	\]      
	In both cases, Proposition~\ref{prop:Renorm} applies and we find \(\sharp\{\text{edges of }\pi\}\geq \eta|x-y|\), where 
	\(\eta=C_1/\sqrt{d}\). Note that, although we may not have centred the boxes in question at \(x\), the box \(\Lambda_L\) is only a small box of constant size close to the origin. As \(\x\in\Lambda_L\) and \(n\) is large, the proof of Proposition~\ref{prop:Renorm}, given in Section~\ref{sec:proofAux} below, still applies. As a result, \(\mathcal{D}_{L}^\eta(m)\) always occurs, if the boxes \(B_{n-1},B_n,\dots\) are all good. Consequently, \(\mathcal{D}_L^\eta(m)\) can only fail to occur if at least one of these boxes is bad. Hence, {by Corollary~\ref{corol:badBoxes}, we infer for some \(c>4d+4|\xi\vee(d+\mu)|\) that}
	\begin{equation}\label{eq:rateFak}
		\begin{aligned}
			\P\big(\bigcup_{k\geq n-1}\{B_k \text{ is bad}\}\big) \leq C (n!)^{-2|\xi\vee (d+\mu)|+c/n},
		\end{aligned}
	\end{equation}
	for large enough \(n\) and some constant \(C>0\). By definition, we have {\(m<K_{n+1}\) and therefore \(m/(n+1)^2\leq K_n\). Furthermore, applying Stirling's formula to large \(n\), we have \((n+1)^2 \leq \log m\), therefore \(m/\log(m)\leq (n!)^2\), and \(\log^\epsilon  m\leq n\), for any \(\epsilon\in(0,1)\).} Plugging this back into~\eqref{eq:rateFak}, we obtain
	\[
		\P(\neg \mathcal{D}_{L}^\eta(m))\leq \P\big(\bigcup_{k\geq n-1}\{B_k \text{ is bad}\}\big) \leq C \Big(\frac{m}{\log m}\Big)^{-|\xi\vee (d+\mu)|+c/(2\log^\epsilon m)}.
	\]
	Taking the logarithms, dividing both sides by \(\log(m)\), and sending \(m\to\infty\) concludes the proof of Theorem~\ref{thm:main}.
	\end{proof}
	 	
\section{Proof of auxiliary results} \label{sec:proofAux}
In this section, we give the proofs of the results concerning Berger's renormalisation scheme used in the previous section. We begin with deriving the probability that a box is bad thus proving Lemma~\ref{lem:badSummable}. 
 
\begin{proof}[Proof of Lemma~\ref{lem:badSummable}.] 
        Let us define, for each \(n\in\N\), the event \(\L_n:=\L(K_n,K_{n-1}/100)\), where \(\L(m,n)\) is the defining event of Property~\ref{G:NoLongE}\(^\mu\), {define \(\alpha:=|\xi\vee (d+\mu)|\) and fix some \(c>4d+4\alpha\).} We immediately infer
        \begin{equation}\label{eq:PL}
        	\P(\L_n)\leq 100^{|\mu|} C_\L K_n^{d}K_{n-1}^{-|\mu|} \leq 100^{|\mu|} C_\L n^{2d}K_{n-1}^{{}-\alpha}.
        \end{equation}
        Observe that this implies, for each fixed \(n\), that \(\P(\L_n)\downarrow 0\), as \(K\to\infty\). Let us further derive, for \(n\geq 2\), a recursive formula for the probability of the stage-\(n\) box being bad. Such a box is bad, if one of the \(3^d\) translated boxes \(B_n^\mathbf{j}\) fails either Property~(a) or~(b). A union bound over the \(3^d\) many boxes \(B_n^\mathbf{j}\) combined with translation invariance \ref{G:Translation}, thus yields
         \[
            \psi_K(n) \leq 3^d\big(\P(\L_n) + \P(\text{at least }3^d+1 \text{ stage-}(n-1) \text{ boxes contained in } B_n \text{ are bad})\big).
         \]
        By construction, if there are at least \(3^d+1\) bad stage-\((n-1)\) boxes contained in \(B_n\), then at least two of these are at distance \(2K_{n-1}\), see Figure~\ref{subfig:badSubbox}. Since there are no more than \(\binom{n^{2d}}{2}\leq n^{4d}\) many possibilities to choose two such boxes, we infer with a union bound, the mixing property~\ref{G:Pmix}\(^\xi\), and translation invariance~\ref{G:Translation},
        \begin{equation} \label{eq:recursion}
        	\begin{aligned}
        	   \psi_K(n) & \leq 3^d \P(\L_n) + 3^d n^{4d}\big(\psi_K(n-1)^2+C_\text{mix}K_{n-1}^{-|\xi|}\big).
        	\end{aligned}
        \end{equation}
        We now prove the claim by induction. 
        Set \(C_2=3^d(100^{|\mu|}C_\L +1+ C_\text{mix})\), recall \(\alpha=|\xi\vee(d+\mu)|\), and let \(N\) be the smallest integer such that {\(|2\alpha|>c/N\)} and
        \[
        	 C_2 (n+1)^{-4\alpha+4d}(n-1)!^{-c/n}\leq 1, \qquad \text{ for all } n\geq N,
        \]
        noting that this \(N\) exists {as \(c>4d+4\alpha\)}. For \(n=1\), we now have
        \[
        	\psi_K(1)\leq 3^d\P(\L_1)\leq (2!)^{-2\alpha+c},
        \]
        for large enough \(K\). Similarly, for \(n=2\), we find using~\eqref{eq:recursion} and \(\psi_K(1)\leq 3^d\P(\L_1)\),
        \begin{equation*}
        	\begin{aligned}
        		\psi_K(2) 
        		& 
        			\leq 3^d \P(\L_2) + 3^d\cdot  2^{4d}\big(\psi_K(1)^2+ C_\text{mix} K_1^{-|\xi|}\big)
        			\leq 3^d \P(\L_2)+3^{2d}\cdot 2^{4d} \, \P(\L_1)^2+C_\text{mix}K_1^{-|\xi|}
        		\\ &
        			\leq (3!)^{-2\alpha+c/2},
        	\end{aligned}
        \end{equation*}   
        perhaps increasing \(K\) if necessary. We proceed similarly \(N\) times, adapting \(K\) at each step if needed and infer \(\psi_K(n)\leq (n+1)!^{-2\alpha+c/n}\) for all \(n=1,\dots,N\). Note that \(K\) is finite still and remains unchanged from this point onwards. Now assume that the claim has been proven until some \(n-1\geq N\). Then, applying~\eqref{eq:PL} and the induction hypothesis to~\eqref{eq:recursion}, we infer 
        \begin{equation*}
        	\begin{aligned}
            	\psi(n) 
            	& 
            		\leq 3^d \, 100^{|\mu|} C_\L n^{2d}K^{-\alpha}(n-1)!^{{}-2\alpha}+3^d n^{4d}\big({(n!^{-2\alpha+c/n})^2}+C_{\text{mix}}K^{-|\xi|}(n-1)!^{-2|\xi|}\big) 
            	\\ &
            		{\leq (n+1)!^{-2\alpha+c/n}\big(C_2 (n+1)^{4d+4\alpha}(n-1)!^{-c/n}\big)\leq (n+1)!^{-2\alpha+c/n}}
            \end{aligned}
        \end{equation*}
        since \(n>N\). This concludes the proof. 
    \end{proof}
    
    \begin{proof}[Proof of Corollary~\ref{corol:badBoxes}.]
    	The result of Corollary~\ref{corol:badBoxes} is a direct consequence of Lemma~\ref{lem:badSummable} and Stirling's formula. More precisely, again with \(\alpha=|\xi\vee (d+\mu)|\),
    	\begin{equation*}
    		\begin{aligned}
				\sum_{k\geq n-1} \psi_K(k)
				& 
					\leq \sum_{k=n}^\infty {(k!)^{-2\alpha+c/k}}
					\leq C	\sum_{k=n}^\infty \big(\sqrt{k} \, e^{k (\log k-1)}\big)^{-2\alpha+c/n} 
				\\ &
					\leq C  (\sqrt{n})^{-2\alpha+c/n} \sum_{k=n}^\infty e^{(-2\alpha+c/n)k(\log n - 1)} 
				\\ &
					\leq C \frac{(\sqrt{n})^{-2\alpha+c/n}}{\log n}e^{(-2\alpha+c/n) n(\log n -1)} 
					\leq C (n!)^{-2\alpha+c/n}, 	
			\end{aligned}
		\end{equation*}
 		for large enough \(n\), where \(C>1\) is a constant that may have changed from line to line. 		
    \end{proof}
   
    We finally turn to the proof of Proposition~\ref{prop:Renorm}. To this end, we first prove an adaptation of~\cite[Lemma~2]{berger2004}. {To appropriately formulate it, we define the number of edges of a path \(\pi\) as }
    \[
    	\ell(\pi):=\sharp\{\text{edges of }\pi\}.
    \]
    
    \begin{lemma}\label{lem:renormPraep}
    	Let \(n\geq N\geq (2d+1)9^d\), as in Proposition~\ref{prop:Renorm}. There exists a universal constant \(C'=C'(d,N)>0\) such that, if \(B_n'\) is some good box with \(\x,\y\in B_n'\) and \(|x-y|>K_{n}/16\), then every path \(\pi\) connecting \(\x,\y\) inside \(B'_n\) has
    	\begin{equation} \label{eq:renormalisationInd}
    		\ell(\pi) \geq C' \Pi(n) |x-y|, \text{ where } \Pi(n):=\prod_{h=N}^n \big(1-\tfrac{N}{h^2}\big).
    	\end{equation}
    \end{lemma}
    \begin{proof}
    	We prove the lemma by induction and we start with the \emph{base case} \(N=n\). Since \(B'_N\) is good, there is no edge contained in it longer than \(K_{N-1}/100\). Furthermore, as \(|x-y|>K_{N}/16\), we infer 
    	\[
    		\ell(\pi)\geq \frac{100|x-y|}{K_{N-1}} \geq \frac{25}{4} N^2. 
    	\]
    	Additionally, \(|x-y|\leq \sqrt{d}K_N\). Thus, the claim follows for \(C'=1/(\sqrt{d}K_{N-1})\) as this implies \((25/4) N^2\geq C'(1-1/N)|x-y|\).  
    	
    	For the \emph{induction step}, assume that \eqref{eq:renormalisationInd} holds for all stages up to stage \(n-1\geq N\). Let \(\x,\y\in B'_n\) with \(|x-y|>K_n/16\) and \(B'_n\) be good. Let \(\pi=(\boldsymbol{v}_0,\dots,\boldsymbol{v}_\ell)\) be a path connecting \(\x=\boldsymbol{v}_0\) and \(\y=\boldsymbol{v}_\ell\) that is completely contained in \(B'_n\). Since \(B'_n\) is good, \(|\boldsymbol{v}_i-\boldsymbol{v}_{i-1}|\leq K_{n-1}/100\). Moreover, at most \(3^d\) of the stage-\((n-1)\) sub boxes in each translation \((B_n')^\mathbf{j}\) are bad. Hence, there are no more than \(9^d\) bad sub boxes in the union of all translations. We denote the bad sub boxes by \(Q_1,\dots, Q_j\) (for \(j\leq 9^d\)) and set \(Q:= Q_1\cup\dots\cup Q_j\). We decompose \(\pi\) into alternating \emph{good segments} \(\pi_s\) and \emph{bad segments} \(\sigma_t\), some of which may be empty, such that \(\pi=(\pi_1,\sigma_1,\pi_2,\dots,\sigma_T,\pi_S)\) for some \(S,T\), where the last vertex in each non empty segment is also the first vertex in the non empty subsequent one. As, by construction, \(S-1\leq T\leq 9^d\), we may simplify \(S=T=9^d\) in the following as this may only introduce additional empty segments at the end of the decomposition. With the same reasoning, we may always start the decomposition with a good segment because, if the first segment is bad, then simply \(\pi_1=\emptyset\). We use the following procedure to obtain the desired decomposition. First, if \(\pi\cap Q=\emptyset\), we simply choose \(\pi_1=\pi\) and all the other segments to be empty. Otherwise, define \(i_1=\min\{i\leq \ell:v_i\in Q\}\) the index of the first vertex contained in the bad region and \(j_1\) the index of the bad box it is contained in, i.e.\ \(v_{i_1}\in Q_{j_1}\), where we choose the smallest such index if two or more bad boxes overlap. However, our result does not depend on the precise ordering of the bad regions and the result holds verbatim for any other decision rule. We set \(\pi_1=(\boldsymbol{v}_0,\dots,\boldsymbol{v}_{i_1-1})\) (or \(\pi_1=\emptyset\) if \(i_1=0\)). Define further \(k_1=\max\{i: v_i\in Q_{j_1}\}\) and set \(\sigma_1=(\boldsymbol{v}_{i_1-1},\dots,\boldsymbol{v}_{k_1+1})\). Let us remark that this segment may also leave and reenter the bad box \(Q_{j_1}\) multiple times but the path never returns to \(Q_{j_1}\) after the vertex with index \(k_1\). Inductively, define \(i_s=\min\{i>k_{s-1}:v_i\in Q\}\), \(j_s=\min\{i:v_{i_s}\in Q_i\}\), and \(k_s=\max\{i:v_{i}\in Q_{j_s}\}\), as well as the good segment \(\pi_s=(\boldsymbol{v}_{k_{s-1}+1},\dots, \boldsymbol{v}_{i_s-1})\), and the bad segment \(\sigma_s=(\boldsymbol{v}_{i_s -1},\dots,\boldsymbol{v}_{k_s+1})\) all the way up to \(S\) and \(T\), {see Figure~\ref{fig:path}.} 
    	
\begin{figure}
	\begin{center}
		\resizebox{0.4\textwidth}{!}{
			\begin{tikzpicture}[every node/.style={scale=0.6}, x=1cm, y=1cm]
  				\tikzset{
    				box/.style   ={thick},
   					path/.style  ={very thick, line cap=round},
    				obsfill/.style ={fill=gray!50},
    				obsline/.style ={thick},
    				lab/.style   ={font=\footnotesize}
  				}

  				\def\L{10}  
  				\draw[box] (0,0) rectangle (\L,\L);

  				\def\s{0.5}                
  				\coordinate (O) at (5.6,4.0); 

  				\path[obsfill] (O) rectangle ++(\s,\s);
  				\draw[obsline] (O) rectangle ++(\s,\s);
 		 		\node[] at ($(O)+(\s /2, \s/2)$) {$Q_1$};
  				\path[obsfill] ($(O)+(\s,0)$)  rectangle ++(\s,\s);
  				\node[] at ($(O)+(\s,0)+(\s /2, \s/2)$) {$Q_3$};
  				\path[obsfill] ($(O)+(0,\s)$) rectangle ++(\s,\s);
  				\node[] at ($(O)+(0,\s)+(\s /2, \s/2)$) {$Q_2$};

  				\path[name path=obsborder, obsline]
    							(O) -- ++(2*\s,0) -- ++(0,\s) -- ++(-\s,0) -- ++(0,\s) -- ++(-\s,0) -- cycle;

  				\draw[obsline] (O) -- ($(O)+(2*\s,0)$) -- ++(0,\s) -- ++(-\s,0)
                 				-- ++(0,\s) -- ++(-\s,0) -- cycle;

  				\path[obsfill] (1.0,7.6) rectangle +(0.5,0.5);
  				\draw[obsline] (1.0,7.6) rectangle +(0.5,0.5);
  				\node[] at ($(1.0,7.6)+(0.25,0.25)$) {$Q_4$};

  				\coordinate (S) at (1.2,1.0);   
  				\coordinate (T) at (9.7,9.0);    

  				\path[name path=curveFull]
    				(S) to[out=15,in=230] (3.0,1.8)
        				to[out=50,in=200] (4.9,3.3)
        				to[out=20,in=220] (5.5,3.7)
        				to[out=40,in=240] (6.8,4.7)
        				to[out=60,in=260] (7.9,6.1)
        				to[out=80,in=270] (9.8,7.5)
        				to[out=85,in=185] (T);

  				\path[name intersections={of=curveFull and obsborder, by={Iin,Iout}}];

  				\draw[path]
    				(S) to[out=15,in=230] (3.0,1.8)
        				to[out=50,in=200] (4.9,2.7)
        				to[out=20,in=220] (5.8,3.7)
        				to[out=40,in=240] (Iin);

  				\draw[path]
    				(Iout) to[out=60,in=260] (7.9,6.1)
           				to[out=80,in=270] (9.8,7.5)
           				to[out=85,in=185] (T);

  				\node[lab, left, scale =1.5]  at (S) {$x$};
  				\node[lab, right, scale =1.5] at (T) {$y$};
			\end{tikzpicture}	
		}
	\end{center}
	\caption{A path connecting \(x\) and \(y\) inside a large good box, with the bad sub boxes in grey. The path decomposes into a good segment \(\pi_1\) from \(x\) to the last vertex before entering \(Q_1\); the bad segment \(\sigma_1\), which is not shown in the picture, consisting of the path segment between the first entrance to \(Q_1\) and the last exit of \(Q_1\); as the path directly enters \(Q_3\) after leaving \(Q_1\), the good segment \(\pi_2\) is empty; the bad segment \(\sigma_2\), again not shown, consists of the path segment between the first entrance to \(Q_3\) and the last exit of \(Q_3\); the final good segment \(\pi_3\) then connects the vertex after leaving \(Q_3\) to \(y\). As the depicted part of the path is completely contained in good regions, no long edges can be used. }
	\label{fig:path}
\end{figure}
    	    	
    	Denote for a path \(\rho\) by \(\mathsf{dist}(\rho)\) the Euclidean distance between its two endpoints. The triangle inequality gives \(|x-y|\leq \sum_{1}^{9^d}\mathsf{dist}(\pi_s)+\sum_1^{9^d}\mathsf{dist}(\sigma_t)\). Since each \(Q_{j_s}\) (a bad stage-\((n-1)\) box) has diameter \(\sqrt{d}K_{n-1}\) and all edge lengths are bounded by \(K_{n-1}/100\), the contribution of any bad segment is bounded by    
    	\begin{equation*}
    		\begin{aligned}
    			\mathsf{dist}(\sigma_t) & \leq |v_{i_t-1}-v_{i_t}| + |v_{i_t}-v_{k_t}|+|v_{k_t}-v_{k_t+1}| \\& \leq K_{n-1}/100 + \sqrt{d}K_{n-1}+K_{n-1}/100 \leq 2 d K_{n-1}.
    		\end{aligned}
    	\end{equation*}
    Let \(I=\{s: \mathsf{dist}(\pi_s)>K_{n-1}/2\}\). Then trivially, \(\sum_{s\not\in I}\mathsf{dist}(\pi_s)\leq 9^d K_{n-1}/2\).
	Combined with the assumption \(|x-y|>K_n/16\) and the fact \(K_n=n^2 K_{n-1}\), we infer by use of the triangle inequality
	\begin{equation}\label{eq:distBoundRenorm}
		\begin{aligned}
			\sum_{s\in I}\mathsf{dist}(\pi_s) & \geq |x-y|- \sum_{t=1}^{9^d}\mathsf{dist}(\sigma_t)-\sum_{s\not\in I}\mathsf{dist}(\pi_s)\geq |x-y| - 2d \cdot  9^d K_{n-1} - \tfrac{9^d}{2} K_{n-1} 
			\\ &
			\geq |x-y|\big(1-\tfrac{N}{n^2}\big)  
		\end{aligned}
	\end{equation}
	by our choice of \(N\). 
	To conclude the proof, assume that for each \(s\in I\)
	\begin{equation}\label{eq:endRenormalisationID}
		\ell(\pi_s)\geq \mathsf{dist}(\pi_s)\cdot C' \Pi(n-1).
	\end{equation}
	{Indeed, if this was the case, the claim would follow by summing~\eqref{eq:endRenormalisationID} over the segments of \(I\) and applying~\eqref{eq:distBoundRenorm}. Hence, it remains to verify~\eqref{eq:endRenormalisationID}, which is also the part where we make use of the induction hypothesis (i.e.,~\eqref{eq:renormalisationInd} applies to \(n-1\)).} 
	
	For \(\boldsymbol{v},\w\in\pi_s\), we write \(\pi_s[\boldsymbol{v},\w]\) for the path segment from \(\boldsymbol{v}\) to \(\w\). Observe that there exists a collection \(\boldsymbol{v}_{k_{s-1}+1}=\w_1,\dots,\w_m=\boldsymbol{v}_{i_s -1}\) of vertices of \(\pi_s\), such that for every \(i\)
	\begin{description}
		\item[(1)] \(|w_{i+1}-w_i|>K_{n-1}/16\), and 
		\item[(2)] \(|v-w_i|<K_{n-1}/2\) for all \(\boldsymbol{v}\in \pi_s[\w_i,\w_{i+1}]\). 
	\end{description}
	That is, the \(\w_i\) divide the good segment \(\pi_s\) into sub segments \(\pi_s[\w_i,\w_{i+1}]\) such that the whole sub segment is contained in the ball of radius \(K_{n-1}/2\), centred at \(w_i\), but still a large Euclidean distance is bridged. This sequence can for instance be constructed using a greedy algorithm~\cite[Prop.~2.6]{komjathy2023FPP}. To this end, one assumes inductively that \(\w_1,\dots,\w_i\) have already been found and the segment \(\pi_s\) is not fully covered yet. Then, the remaining path is either contained in the ball of radius \(K_{n-1}/2\), centred in \(\boldsymbol{v}_{i_{s}-1}\) (the last vertex on \(\pi_s\)), in which case we choose \(\w_{i+1}=\boldsymbol{v}_{i_s-1}\), or this is not the case and we follow the path and pick \(\w_{i+1}\) to be the first vertex that fulfils both properties. Note that this is always possible as \(\pi_s\) is a good path segment and the assumption implies that either \(\boldsymbol{v}_{i_s-1}\) is at a far distance from \(\w_i\) or the paths wanders far off before reaching \(\boldsymbol{v}_{i_s-1}\). 	Now, by Property~\textbf{(2)} and the choice of \(\pi_s\), there exists a good stage-\((n-1)\) box containing \(\pi_s[\w_i,\w_{i+1}]\). By~\textbf{(1)} and the induction hypothesis~\eqref{eq:renormalisationInd} for \(n-1\), we have
	\[
		\ell(\pi_s[\w_i,\w_{i+1}])\geq \mathsf{dist}(\pi_s[\w_i,\w_{i+1}])\cdot C' \Pi(n-1),
	\]
	using \(|w_i-w_{i+1}|=\mathsf{dist}(\pi_s[\w_i,\w_{i+1}])\). Using \(\ell(\pi_s)=\ell(\pi_s[\w_1,\w_2])+\dots+\ell(\pi_s[\w_{m-1},\w_m])\) together with the triangle inequality, thus proves the claim~\eqref{eq:endRenormalisationID}. This concludes the proof of Lemma~\ref{lem:renormPraep}.
	\end{proof}

	Using Lemma~\ref{lem:renormPraep}, we may now proceed and prove Proposition~\ref{prop:Renorm}.
	
    \begin{proof}[Proof of Proposition~\ref{prop:Renorm}.]	 
	Recall the notation introduced in the previous lemma. Set \(C_3:=C'\Pi(\infty)\), let \(\x,\y\in\scrV\) such that \(|x-y|>K_{n-1}/8\) and \(B_{n-1}(x)\), \(B_n(x)\) are good, and let \(\pi\) be a path connecting \(\x\) and \(\y\) within \(B_n(x)\). {Note that we can apply the decomposition into good and bad segments of the previous proof to \(\pi\). Let us consider two different cases, depending on the distance \(|x-y|\). We start with the case \(|x-y|\geq 2N K_{n-1}\). In that case, we can repeat Calculation~\eqref{eq:distBoundRenorm} to obtain 
	\begin{equation}\label{eq:distBoundForEnd}
		\sum_{s\in I}\mathsf{dist}(\pi_s)\geq |x-y|-N K_{n-1}\geq |x-y|/2,
	\end{equation}
	Furthermore,~\eqref{eq:endRenormalisationID} is a direct consequence of the path decomposition and the proven statement~\eqref{eq:renormalisationInd}. It therefore applies to \(\x,\y\) and either the good box \(B_{n-1}(x)\) or the good box \(B_n(x)\), depending on the precise location of \(\y\). In either case, we infer since \(\Pi(n)\) is decreasing in \(n\),   
	\[
		\ell(\pi)\geq C_3 \sum_{s\in I} \mathsf{dist}(\pi_s) \geq \frac{C_3}{2}|x-y|.
	\]}
	
	If instead \(|x-y|< 2N K_{n-1}\), then by the fact that \(B_n(x)\) is good and no edge is longer than \(K_{n-1}/100\), there must exists a vertex \(\boldsymbol{v}\in \pi\), such that \(K_{n-1}/16<|x-v|\leq K_{n-1}/8\). If \(\boldsymbol{v}\) is the first such vertex on the path, then additionally \(\pi[\x,\boldsymbol{v}]\subset B_{n-1}(x)\). By assumption, the box \(B_{n-1}(x)\) is good and we can apply~\eqref{eq:renormalisationInd} and the bounds on the Euclidean distance to obtain
	\begin{equation*}
		\begin{aligned}
			\ell(\pi)\geq \ell(\pi[\x,\boldsymbol{v}])\geq C_3 |x-v| \geq C_3 \frac{K_{n-1}}{16} \geq C_3 \frac{|x-y|}{32 N}. 
		\end{aligned}
	\end{equation*} 
	Hence, the claim follows for \(C_1=C_3/32N\). 
	\end{proof}

\paragraph{Acknowledgement.} I would like to thank Peter Gracar and Jonas Köppl, for their valuable support and helpful discussions.

\paragraph{Funding information.} This research was supported by the Leibniz Association within the Leibniz Junior Research Group on Probabilistic Methods for Dynamic Communication Networks as part of the Leibniz Competition (grant no. J105/2020).

\section*{References} 
\renewcommand*{\bibfont}{\footnotesize}
\printbibliography[heading = none]
\end{document}